\def\version{28/07/2020 \quad version 6
\hfill\href{http://arxiv.org/abs/1406.5620}{arXiv:1406.5620}
}

\documentclass[11pt,a4paper]{amsart}
\usepackage{fullpage}

\usepackage{amssymb,amscd,amsxtra}

\usepackage{xfrac}

\usepackage[all,pdf]{xy}
\newdir{ >}{{}*!/-8pt/@{>}}

\usepackage{mathrsfs}

\usepackage{braket}

\usepackage[alphabetic,lite]{amsrefs}

\usepackage{hyperref}

\usepackage{pigpen}
\def\PO{\text{\pigpenfont R}}

\renewcommand{\thefootnote}{\fnsymbol{footnote}}
\long\def\symbolfootnote[#1]#2{\begingroup%
\def\thefootnote{\fnsymbol{footnote}}\footnote[#1]{#2}\endgroup}

\newtheorem{thm}{Theorem}[section]

\newtheorem{prop}[thm]{Proposition}

\theoremstyle{definition}
\newtheorem{rem}[thm]{Remark}

\numberwithin{equation}{section}
\numberwithin{figure}{section}

\def\:{\colon}
\def\.{\cdot}

\def\<{\left\langle}
\def\>{\right\rangle}
\def\({\left(}
\def\){\right)}

\def\epsilon{\varepsilon}
\def\phi{\varphi}

\def\leq{\leqslant}
\def\geq{\geqslant}

\def\la{\leftarrow}

\def\Lra{\Longrightarrow}
\def\ra{\rightarrow}

\def\hat#1{\widehat{#1}}
\def\tilde#1{\widetilde{#1}}
\def\iso{\cong}

\def\hotimes{\hat{\otimes}}
\def\hsquare{\hat{\square}}

\DeclareMathOperator{\Cont}{Cont}

\DeclareMathOperator{\Id}{Id}

\DeclareMathOperator{\im}{im}

\def\CP{\mathbb{C}\mathrm{P}}

\def\E{\mathrm{E}}
\def\F{\mathbb{F}}

\def\HP{\mathbb{H}\mathrm{P}}

\def\N{\mathbb{N}}
\def\Q{\mathbb{Q}}

\def\Z{\mathbb{Z}}

\DeclareMathOperator{\Ext}{Ext}

\DeclareMathOperator{\Hom}{Hom}

\DeclareMathOperator{\Map}{Map}
\def\Mapc{\Map^{\mathrm{c}}}

\DeclareMathOperator{\Cotor}{Cotor}
\DeclareMathOperator{\Tor}{Tor}
\def\LCotor{\widehat{\Cotor}}

\DeclareMathOperator*{\colim}{colim}
\DeclareMathOperator*{\hocolim}{hocolim}

\DeclareMathOperator{\dlQ}{Q}
\DeclareMathOperator{\tdlQ}{\hat{Q}}
\def\Einfty{$\mathcal{E}_\infty$ }

\DeclareMathOperator{\Prim}{Prim}

\def\Kriz{{K\v{r}\'{i}\v{z}}}

\title
{Power operations in $\boldsymbol{K}$-theory
completed at a prime}
\author{Andrew Baker}
\date{\version}
\address{
School of Mathematics \& Statistics, 
University of Glasgow,
Glasgow G12 8QQ, Scotland.}
\email{a.baker@maths.gla.ac.uk}
\urladdr{http://www.maths.gla.ac.uk/$\sim$ajb}
\thanks{
The mathematics described in this paper is based
in part on work supported by the National Science
Foundation under Grant No.~0932078~000 while the
author was in residence at the Mathematical Sciences
Research Institute in Berkeley California, during
the Spring~2014 semester; the author also
acknowledges the support of the Max Planck Institute
for Mathematics in Bonn during a visit in January~2020. \\
I would like to thank Gerd Laures, Justin Noel
and Charles Rezk for helpful conversations, and
especially Francis Clarke who taught me about
the important r\^ole of $p$-adic analysis in
$K$-theory back in the early~1980s.}
\keywords{$K$-theory, $\E_\infty$ ring spectrum,
commutative $S$-algebra, power operation}
\subjclass[2010]{Primary 55P43;
Secondary 13D03, 55N35, 55P48}

\begin{document}

\begin{abstract}
We describe the action of power operations
on the $p$-completed cooperation algebras
$K^\vee_0 K = K_0(K)\sphat_p$\/ for $K$-theory
at a prime~$p$, and $K^\vee_0 KO = K_0(KO)\sphat_2$.
These results are used to identify the $K(1)$-local
homotopy type of some $\E_\infty$ ring spectra
obtained by killing elements of Hopf invariant~$1$.
\end{abstract}

\maketitle

\tableofcontents

\section*{Introduction}

Power operations in suitably completed (co)homology
theories have been studied and used by several authors,
for example Rezk~\cites{CR:CongCond,CR:ModIsogCmplxes,CR:PowOps-Koszul};
the paper of Barthel and Frankland~\cite{TB&MF} building
on work of McClure~\cite{LNM1176} provides a convenient
account of this, in particular for the case of $p$-complete
$K$-theory. An important source on related mathematics
is the article by Hopkins~\cite{MJH:K(1)localEinfty},
and indeed the volume~\cite{TMF} contains much that
the reader may find helpful.

In the present paper we describe the action of the
$\theta$-operator (which we follow~\cite{TB&MF} in
denoting by~$\dlQ$) on the $p$-completed cooperation
algebra
\[
K^\vee_0 K = K_0(K)\sphat_p
           = \pi_0(L_{K(1)}(K\wedge K)),
\]
where $K=KU$. We expect this to be of use in 
investigating the $\theta$-action and its 
interaction with the $K^\vee_*(K)$-coaction 
on $K^\vee_*(A)$ for any \Einfty{} ring 
spectrum~$A$. We also give some results 
on $K^\vee_0(KO)$ when~$p=2$ and on 
$K^\vee_*(\mathbb{P}X)$, where $\mathbb{P}X$ 
denotes the free commutative $S$-algebra 
on a spectrum~$X$ introduced in~\cite{EKMM}.

It is likely that some of our results are known
to experts, but we have not found a published
source, so we feel it worthwhile writing them
down.

An obvious related problem to investigate is
that of describing the actions of power operations
on $K^\vee_0(BU)$ or equivalently on $K^\vee_0(MU)$
(these actions correspond under the Thom isomorphism).
The \Einfty{} orientation of~\cite{MJ:Glasgow}
induces a morphism of $\theta$-algebras
$K^\vee_0(MU)\to K^\vee_0(K)$ but this is not
injective on the image of the Hopf algebra
primitives $\Pr K^\vee_0(BU)$, and this seems
to make the determination of the action on
primitives more delicate then in the case of
ordinary mod~$p$ homology as carried out by
Kochman~\cite{SOK:DLops}. We may return to
this in future work.

\bigskip
\noindent
\textbf{Conventions and notation:} We will work
with \Einfty{} ring spectra in the setting of
commutative $S$-algebras of~\cite{EKMM} and use
these terms interchangeably. We will assume that
$KU$ and $KO$ have their standard \Einfty{}
ring structures as produced in~\cite{HGamma}
for example.

Throughout, $p$ will be a fixed prime and $K=KU_{(p)}$
will denote the $p$-local $2$-periodic complex
$K$-theory ring spectrum; we will also denote
the $p$-adic completion of~$K$ by $K\sphat_p=KU\sphat_p$.
We will often denote (co)homology without
brackets where appropriate by setting
$K^*X=K^*(X)$ and $K_*X=K_*(X)$ for example,
but include brackets where it improves
readability.

\section{$L$-complete modules}\label{sec:L-complete}

We will be working with $p$-complete $K$-theory
for a prime~$p$, and this takes values in
the category of $L$-complete graded modules
for the local ring $\Z_{(p)}$. The utility
of working with such a category originated
in work of Greenlees \& May \cite{JPG&JPM:I-adic}
and was made explicit by Hovey \& Strickland~\cite{MAMS666}.
The reader is also referred to Barthel \&
Frankland~\cite{TB&MF} for a more recent
account.

A fundamental observations is that for any
spectrum each $p$-completed $K$-theory group
\[
K^\vee_nX = \pi_n(L_{K(1)}(K\wedge X))
\]
is $L$-complete (with respect to~$\Z_{(p)}$),
i.e., $K^\vee_nX\iso L_0K^\vee_nX$ where $L_s$
($s\geq0$) is the left derived functor of
$p$-adic completion on the category
of~$\Z_{(p)}$-modules. In fact $L_s$ is trivial
when $s>1$.

When $M$ is $\Z_{(p)}$-free or flat then
$L_0M = M\sphat_p$\; and $L_1M=0$ by~\cite{L-complete}.
More generally, $L_*M$ can calculated by taking
a free resolution
\[
0\la M \leftarrow F_0 \leftarrow F_1 \la 0
\]
and taking homology of the induced complex
\[
0\la (F_0)\sphat_p \leftarrow (F_1)\sphat_p \la 0.
\]
For $M=K_n(X)$ this allows us to induce up the
effect of a natural transformation $\theta\:K_n(-)\to K_n(-)$.
To see how to do this we need some background.

Recall that a ring spectrum $E$ satisfies the
\emph{Adams condition} of~\cite{JFA:BlueBook}
if it can be written as colimit
$E=\colim_\alpha E_\alpha$ of dualisable
spectra $E_\alpha$. This condition ensures the
existence of suitable resolutions for constructing
Universal Coefficient spectral sequences.

In particular, $KU$ and $KO$ satisfy the Adams
condition, see~\cite{JFA:BlueBook}*{proposition~13.4}.
The proof there uses even suspensions of skeleta
of $BU$ and $BSp$ (with cells in even degrees);
in fact these can be replaced by suspensions
of skeleta of $\CP^\infty$ and $\HP^\infty$
by results of~\cite{AHS}.

Then the $K_*$-module $K_*(X)$ can be resolved
using the following procedure due to Adams,
see~\cite{JFA:BlueBook}*{lemma~13.7}. Take a
set of $K_*$-module generators of
$K_*(X)=\colim_\alpha\pi_*(K_\alpha\wedge X)$
and form their adjoint maps $f\:\Sigma^{n(f)}DE_\alpha\to X$
so that together these induce an epimorphism
\[
\bigoplus_f K_*(DE_\alpha) =
K_*\biggl(\bigvee_f DE_\alpha \biggr)
\xrightarrow{\epsilon}K_*X.
\]
Here each $K_*(DE_\alpha)$ is a finitely
generated free $K_*$-module and by work of
Hovey~\cite{Hovey:L-colimits}*{theorem~3.3},
\[
K^\vee_*\biggl(\bigvee_f DE_\alpha\biggr)
\iso
\biggl(\bigoplus_f K_*(DE_\alpha)\biggr)\sphat_p.
\]
which is pro-free. As $K_*$ is a graded
principal ideal domain, $\ker\epsilon$ is
also a free $K_*$-module, so $L_*K_*(X)$
can be calculated using the complex
\[
0\la K^\vee_*\biggl(\bigvee_f DE_\alpha\biggr)
\leftarrow
(\ker\epsilon)\sphat_p\la 0.
\]
Notice also that the spectral sequence of
\cite{MH:SSMoravaEthy}*{corollary~3.2}
collapses to give a collection of short
exact sequences
\[
0\ra L_0K_n(X) \to K^\vee_*(X) \to L_1K_{n-1}(X)\ra 0.
\]

\section{$K$-theory completed at a prime and
power operations}\label{sec:K-thy&powops}

We first recall some standard facts about the
rings of $p$-local integers~$\Z_{(p)}$ and
$p$-adic integers~$\Z_p$. By definition, if
we give $\Z_{(p)}$ and $\Z_p$ the $p$-adic
norm topologies then $\Z_{(p)}\subseteq\Z_p$
is a dense subring. The residue fields of
$\Z_{(p)}$ and $\Z_p$ both agree with the
finite field $\F_p$ which we give the discrete
topology. There is a pullback square of
topological multiplicative monoids
\[
\xymatrix{
\Z_{(p)}^\times\ar@{^{(}->}[r]\ar@{->>}[d] & \Z_{(p)}\ar@{->>}[d] \\
\F_p^\times\ar@{^{(}->}[r] & \F_p
}
\]
and on $p$-adic completion this becomes
the pullback square
\[
\xymatrix{
\Z_p^\times\ar@{^{(}->}[r]\ar@{->>}[d] & \Z_p\ar@{->>}[d] \\
\F_p^\times\ar@{^{(}->}[r] & \F_p
}
\]
so $\Z_p^\times$ is the completion of
$\Z_{(p)}^\times$ with respect to the
$p$-adic norm.

It is known from~\cites{AHS,JFA&FWC,padic,In-localEn}
that
\[
K_0K \iso
\{f(w)\in\Q[w,w^{-1}] : f(\Z_{(p)}^\times)\subseteq\Z_{(p)}\},
\]
and $K_0K$ is a free $\Z_{(p)}$-module. Since
$\Z_{(p)}^\times$ is a dense subgroup of $\Z_{p}^\times$,
we may interpret Laurent polynomials as continuous
functions on $\Z_{p}^\times$ and obtain
\[
K_0K \iso
\{f(w)\in\Q[w,w^{-1}] : f(\Z_{p}^\times)\subseteq\Z_{p}\}
\subseteq\Cont(\Z_{p}^\times,\Z_{p}),
\]
where the latter is the $p$-adic Banach algebra
of continuous maps $\Z_{p}^\times\to\Z_{p}$
equipped with the operator norm; it is known
that this subring of $\Cont(\Z_{p}^\times,\Z_{p})$
is dense. It follows that
\[
K^\vee_0K = \pi_0((K\wedge K)\sphat_{p}) = (K_0K)\sphat_{p},
\]
where the $p$-adic topology involved in the completion
agrees with $p$-adic norm topology inherited from
$\Cont(\Z_{p}^\times,\Z_{p})$. Therefore there is
an isomorphism of $p$-adic Banach algebras
\begin{equation}\label{eq:KvK-Cont}
K^\vee_0K \iso \Cont(\Z_{p}^\times,\Z_{p}).
\end{equation}
For $a\in\Z_{(p)}^\times$, the stable Adams operation
\[
\psi^a\in K^0K \iso \Hom_{\Z_{(p)}}(K_0K,\Z_{(p)})
\]
is determined by the pairing
$\braket{ - | - }\:K^0K\otimes K_0K\to\Z_{(p)}$,
i.e.,
\[
\braket{\psi^a | f(w)} = f(a).
\]
This extends to a continuous pairing given by
\[
\braket{\psi^a | f} = f(a)
\]
if $a\in\Z_{p}^\times$ and $f\in\Cont(\Z_{p}^\times,\Z_{p})$;
here $\psi^a$ is best viewed as an element of
the pro-group ring
\[
\Z_p[\![\Z_{p}^\times]\!] \iso (K^0K)\sphat_p
\]
For more details on $K_0(K)$ and $\Cont(\Z_{p}^\times,\Z_{p})$,
see~\cite{HGamma}*{section~3}; for a broader overview of the
connections with $p$-adic analysis see~\cite{FWC:UltrametAnal}.

We also recall that $K_0K$ is a bicommutative $\Z_{(p)}$-Hopf
algebra with coproduct $\Psi$ given by
\[
\Psi(f(w)) = f(w\otimes w)
\]
and antipode $\chi$ given by
\[
\chi(f(w)) = f(w^{-1}).
\]
Using the linear pairing $\braket{-|-}$ we can
obtain a left action of~$K^0K$ on $K_0K$; for
$\alpha\in K^0K$, we write $\alpha f(w)$ for
this. In particular, if $a\in\Z_{(p)}^\times$
this coincides with the action of the Adams
operation $\psi^a$,
\[
\psi^a f(w) = f(a^{-1}w).
\]
The reason for the inverse is that we are using
the standard left action of the dual of the
Hopf algebra $K_0K$ defined by
\[
\alpha x = \sum_i\braket{\alpha(\chi(x'_i))|x''_i},
\]
where $\Psi x = \sum_i x'_i\otimes x''_i$,
$\Psi(g(w))=g(w\otimes w)$ and $\chi(g(w))=g(w^{-1})$.

In the $p$-complete setting, (stable) Adams operations
are indexed by the $p$-adic units $\Z_p^\times\subseteq\Z_p$.
It follows that there is a continuous action
\[
\Z_p^\times \times {K_r(X)_p\sphat}\to{K_r(X)_p\sphat}\; ;
\quad
(\alpha,x) \mapsto \psi^\alpha(x).
\]

We use notation from \cite{LNM1176}*{chapter~IX} and
the more recent~\cite{TB&MF}. For an \Einfty{} ring
spectrum $A$ there is a natural power operation
$\dlQ\:K^\vee_0A\to K^\vee_0A$ (sometimes also
called $\theta$) satisfying properties that can be
deduced from those listed in~\cite{LNM1176}*{theorem~IX.3.3}
for the homology theories $K_*(-;p^r)$ with coefficients,
and are discussed in~\cite{TB&MF}*{section~6}, although
the version there is for $\Z/2$-graded $K$-theory.
However, as we are mainly interested in the case of
$K^\vee_*K$ which is concentrated in even degrees,
we work mostly with $K^\vee_0(-)$ but sometimes
need to relate this to $K^\vee_{2n}(-)$ for an
integer~$n$.

The operation $\dlQ$ is neither additive nor
multiplicative, but it satisfies the identities
\begin{subequations}\label{eq:Q-Properties}
\begin{align}
\dlQ(x+y) &= \dlQ x + \dlQ y
   + \frac{1}{p}\biggl(x^p + y^p - (x+y)^p\biggr),
\label{eq:Q-Properties-add} \\
\dlQ(xy) &= y^p\dlQ x + x^p\dlQ y + p\dlQ x\dlQ y,
\label{eq:Q-Properties-mult}
\end{align}
\end{subequations}
or equivalently the operation $\tdlQ$ defined
by
\[
\tdlQ x = p\dlQ x + x^p
\]
is additive and multiplicative,
\begin{align*}
\tdlQ(x+y) &= \tdlQ x + \tdlQ y, \\
\tdlQ(xy) &= \tdlQ x \tdlQ y.
\end{align*}
We also have $\dlQ1=0$, hence $\tdlQ1=1$ and
$\tdlQ$ is a (unital) ring homomorphism.
Finally, if $a\in\Z_{(p)}$ and $u\in\Z_{(p)}^\times$,
\begin{align*}
\dlQ(ax) &= a\dlQ(x) + \frac{(a-a^p)}{p} x^p, \\
\tdlQ(ax) &= a\tdlQ x,  \\
\psi^u\dlQ(x) &= \dlQ(\psi^u x).
\end{align*}
When $K^\vee_r(A)=K_r(A)\sphat_p$, the operations
$\dlQ$ and $\tdlQ$ are continuous with respect to
the $p$-adic topology. This allows us to extend
these identities to the case where
$\alpha\in\Z_p^\times$,
\begin{align*}
\dlQ(\alpha x)
&= \alpha\dlQ(x) + \frac{(\alpha - \alpha ^p)}{p} x^p, \\
\psi^\alpha\dlQ(x) &= \dlQ(\psi^\alpha x), \\
\tdlQ(\alpha x) &= \alpha\tdlQ x, \\
\psi^\alpha\tdlQ(x) &= \tdlQ(\psi^\alpha x).
\end{align*}

Suppose that $X$ is an infinite loop space (and
so $\Sigma^\infty_+X$ is an \Einfty{} ring
spectrum). If $K_0(\Sigma^\infty_+X)$ is
$\Z_{(p)}$-free so that
$K^\vee_0(\Sigma^\infty_+X)=K_0(\Sigma^\infty_+X)\sphat_p$
is pro-free, the diagonal map on~$X$ induces a
coalgebra structure on $K_0(\Sigma^\infty_+X)$
and a topological coalgebra structure on
$K^\vee_0(\Sigma^\infty_+X)$. In that situation,
$\tdlQ$ is a coalgebra morphism; in particular,
$\tdlQ$ preserves coalgebra primitives.

We also mention a useful fact about Adams operations.
Let $\alpha\in\Z_{p}^\times$ and suppose that
$\psi^\alpha x=\alpha^dx$. Since $\psi^\alpha$ is
a ring homomorphism,
\begin{align*}
\psi^\alpha\tdlQ x
&= p\dlQ(\psi^\alpha x) + (\psi^\alpha x)^p \\
&= p\dlQ(\alpha^dx) + (\alpha^d x)^p \\
&= \tdlQ(\alpha^d x),
\end{align*}
giving the identity
\[
\psi^\alpha\tdlQ x = \alpha^d\tdlQ x.
\]

\section{Power operations on $K^\vee_0K$
and on $K^\vee_0KO$ for $p=2$}\label{sec:Powops-KK}

For the case of $K^\vee_0K$ we continue to
assume that~$p$ is an arbitrary prime.

We begin with the action of $\dlQ$ on the
basic element $w\in K_0K\subseteq K^\vee_0K$.
For $a\in\Z_{(p)}^\times$,
\[
\psi^a\dlQ(w) = \dlQ(\psi^a w) = \dlQ(a^{-1}w).
\]
We write $\dlQ(w) = f_0(w)$ where $f_0\in\Cont(\Z_{p}^\times,\Z_{p})$
is the function given by $x\mapsto f_0(x)$,
so we are identifying~$w$ with the inclusion
function $\Z_{p}^\times\to\Z_{p}$ under the
isomorphism~\eqref{eq:KvK-Cont}.

By~\cite{LNM1176}*{theorem~IX.3.3(vi)}, for $k\in\Z$,
\[
\dlQ(kw) = k\dlQ(w) + \frac{(k-k^p)}{p}w^p,
\]
so as $\Z_{(p)}^\times\subseteq\Z_{p}^\times$
is dense, this defines a continuous function
\[
\Z_{p}^\times\times\Z_{p}^\times\to\Z_{p};
\quad
(x,y) \mapsto xf_0(y) + \frac{(x-x^p)}{p}y^p.
\]
Taking $y=1$, this restricts to the continuous
function
\[
\Z_{p}^\times\to\Z_{p};
\quad
x\mapsto xf_0(1) + \frac{(x-x^p)}{p},
\]
and as $f_0(1)=0$, we have
\[
f_0(x)=\frac{(x-x^p)}{p}.
\]
Hence we have
\begin{equation}\label{eq:Qw}
\dlQ w = f_0(w) = \frac{(w-w^p)}{p}.
\end{equation}

For $n\in\N$, by~\cite{LNM1176}*{theorem~IX.3.3(vii)}
\[
\dlQ(w^{n+1}) =
w^p\dlQ(w^{n}) + w^{np}\dlQ(w) + p\dlQ(w^{n})\dlQ(w)
\]
and an easy induction gives the general formula
\[
\dlQ(w^n) = \frac{(w^n-w^{np})}{p}
\]
for all natural numbers. We also have
\[
0 = \dlQ(1) = \dlQ(w^nw^{-n}) =
w^{np}\dlQ(w^{-n}) + w^{-np}\dlQ(w^{n})
    + p\dlQ(w^{n})\dlQ(w^{-n})
\]
and so
\[
\dlQ(w^{-n}) = \frac{w^{-n}-w^{-np}}{p}.
\]
Therefore for all $n\in\Z$,
\begin{equation}\label{eq:Qw^n}
\dlQ(w^n) = \frac{w^n-w^{np}}{p}.
\end{equation}

The operation $\tdlQ$ is given by
\[
\tdlQ(w^n) = \tdlQ(w)^n,
\]
so for any $g\in\Cont(\Z_p^\times,\Z_p)$
we have
\[
\tdlQ(g(w)) = g(\tdlQ w) = g(w),
\]
and therefore
\[
\dlQ(g(w)) = \frac{1}{p}(g(w)-g(w)^p).
\]
This shows that the sequence of polynomial
functions defined recursively by $\theta_0(w) = w$
and for $n\geq1$,
\[
\theta_n(w) =
 \frac{1}{p}(\theta_{n-1}(w)-\theta_{n-1}(w)^p),
\]
is also given by
\begin{equation}\label{eq:Qtheta_n}
\theta_n(w) = \dlQ(\theta_{n-1}(w)).
\end{equation}
It is known that a (topological) $\Z_p$-basis
for $K^\vee_0K$ can be made using monomials
in the $\theta_n(w)$, see~\cite{padic} for
example. One interpretation of what we have
shown is the following result which seems to
have been long known to Mike Hopkins \emph{et al},
but we do not know a published source; a referee
has drawn our attention to Mark Behrens'
article~\cite{TMF}*{chapter~12, section~6} which
contains a related moduli-theoretic interpretation
of such $\theta$-algebras which may lead to similar
results. We interpret the operation $\dlQ$ as a
realisation of an action of~$\theta$ and therefore
$K^\vee_0K$ becomes a $p$-complete
$\Z_p$-$\theta$-algebra~\cites{AKB:padiclambda,TB&MF}.
\begin{prop}\label{prop:K0K-ThetaAlg}
The $p$-complete $\Z_p$-$\theta$-algebra $K^\vee_0K$
is generated by the element~$w$. Hence $K^\vee_0K$
is a quotient of the free $p$-complete
$\Z_p$-$\theta$-algebra $K^\vee_0(\mathbb{P}S^0)$,
namely
\[
K^\vee_0K \iso
\Z_p[\theta^s(w):s\geq0]\sphat_p\biggl/\biggr.
(\!(\theta^s(w)^p-\theta^s(w)+p\theta^{s+1}(w)
                             : s\geq0)\!).
\]
\end{prop}

Here the quotient is taken with respect to the
$p$-adic closure of the ideal generated by the
stated elements, indicated by the use of
$(\!(-)\!)$ rather than $(-)$. This shows that
apart from the $p$-adic completion involved,
$K^\vee_0K$ is a colimit of Artin-Schreier
extensions of the form
\[
\Z_p[X]/(X^p-X+pa)
\]
whose mod~$p$ reduction is the \'etale $\F_p$-algebra
\[
\F_p[X]/(X^p-X) \iso \prod_{0\leq r\leq p-1}\F_p.
\]

Our discussion also shows that the antipode
of $K^\vee_0(K)$, $\chi$ satisfies
\begin{equation}\label{eq:Qchi}
\chi\dlQ = \dlQ\chi.
\end{equation}

Suppose that~$A$ is an \Einfty{} ring spectrum
(or a $K(1)$-local \Einfty{} ring spectrum).
Then we may consider $K^\vee_\bullet(A)$ where
$K^\vee_\bullet(-)$ denotes the $\Z/2$-graded
$p$-complete theory. The power operation~$\dlQ$
intertwines with the coaction as described
in~\cite{Nishida}*{(2.5)}, giving
\begin{equation}\label{eq:psiQ}
\Psi\dlQ x = \dlQ(\Psi x)
\end{equation}
since the antipode $\chi$ satisfies~\eqref{eq:Qchi}
and we have a simpler situation compared to ordinary
mod~$p$ homology where the dual Steenrod algebra
supports two distinct Dyer-Lashof structures related
by the antipode.

We now give a brief description of the modification
required to describe power operations in $K_0^\vee KO$
at the prime~$p=2$. For $KO_*KO_{(2)}$, results
of~\cites{JFA&FWC,AHS} give
\begin{itemize}
\item
for all $m\in\Z$, $KO_mKO_{(2)}\iso KO_m\otimes KO_0KO_{(2)}$;
\item
$KO_0KO_{(2)}$ is a countable free $\Z_{(2)}$-module;
\item
$KO_0KO_{(2)} =
\{f(w)\in\Q[w^2,w^{-2}] : f(\Z^\times_2)\subseteq\Z_2\}$.
\end{itemize}

Passing to $K^\vee_0KO$, recalling that the squaring
homomorphism
\[
\Z^\times_2 = \{\pm1\}\times (1+4\Z_2)
\xrightarrow{\;\;}
1+8\Z_2 \subseteq \Z^\times_2
\]
is surjective, the natural \Einfty{} morphism $KO\to KU$
induces a monomorphism of $2$-complete $\theta$-algebras
$K^\vee_0(KO)\to K^\vee_0(K)$ coinciding with the inclusion
of the continuous functions factoring through~$(-)^2$.

It is clear that $\dlQ$ restricts to $K^\vee_0KO$
and is given by
\[
\dlQ(f) = \frac{(f-f^2)}{2}.
\]
The following elements defined inductively provide
a topological basis for $K^\vee_0KO$:
\[
\Theta_0(w)  = \frac{1-w^2}{8},
\qquad
\Theta_n(w) = \frac{\Theta_{n-1}(w)-\Theta_{n-1}(w)^2}{2}
    \quad (n\geq1).
\]
Then the distinct monomials
$\Theta_0(w)^{\epsilon_0}\Theta_1(w)^{\epsilon_1}
   \cdots\Theta_\ell(w)^{\epsilon_\ell}$
with $\epsilon_j=0,1$ form a topological basis.
Here is the analogue of Proposition~\ref{prop:K0K-ThetaAlg}.
\begin{prop}\label{prop:K0KO-ThetaAlg}
The\/ $2$-complete $\Z_2$-$\theta$-algebra
$K^\vee_0KO$ is a quotient of the free\/
$2$-complete $\Z_2$-$\theta$-algebra generated
by the element\/~$\Theta_0(w)$, i.e.,
\[
K^\vee_0KO \iso
\Z_2[\Theta_s(w):s\geq0]\sphat_2 \,/\,
(\!(\Theta_s(w)^2-\Theta_s(x)+2\Theta_{s+1}(x)
                             : s\geq0)\!).
\]
\end{prop}

\section{The completed $K$-theory of free algebras}
\label{sec:FreeAlgebras}

In this section we will describe $K^\vee_0(\mathbb{P}X)$,
at least for spectra~$X$ for which $K^\vee_0X$ is
suitably restricted. For our purposes, it will suffice
to assume that~$X$ is a CW spectrum with only finitely
many even dimensional cells. It will be useful to
examine how $K^\vee_0(\mathbb{P}X)$ behaves for such
complexes.

Suppose that the $(n-1)$-skeleton $X^{[n-1]}$ of~$X$
is defined. Then the $n$-skeleton $X^{[n]}$ is a
pushout defined by a diagram of the form
\[
\xymatrix{
\bigvee_i S^{n-1}\ar[r]\ar[d]\ar@{}[dr]|{\PO}
 & \bigvee_i D^n\ar[d]  \\
X^{[n-1]}\ar[r] & X^{[n]}
}
\]
for a finite wedge of spheres $\bigvee_i S^{n-1}$.
Similarly there is a pushout diagram of commutative
$S$-algebras
\[
\xymatrix{
\mathbb{P}(\bigvee_i S^{n-1})\ar[r]\ar[d]\ar@{}[dr]|{\PO}
 & \mathbb{P}(\bigvee_i D^n)\ar[d] \\
\mathbb{P}(X^{[n-1]})\ar[r]  & \mathbb{P}(X^{[n]})
}
\]
so $(\mathbb{P}X)^{\langle n\rangle}=\mathbb{P}(X^{[n]})$
is the \Einfty{} $n$-skeleton of the CW commutative
$S$-algebra $\mathbb{P}X$.

If the cells of $X$ are all even dimensional,
we only encounter pushout diagrams of the form
\[
\xymatrix{
\mathbb{P}(\bigvee_i S^{2m-1})\ar[r]\ar[d]\ar@{}[dr]|{\PO}
  & \mathbb{P}(\bigvee_i D^{2m})\ar[d] \\
(\mathbb{P}X)^{\langle 2m-2\rangle}\ar[r]
  & (\mathbb{P}X)^{\langle 2m\rangle}
}
\]
where
\[
(\mathbb{P}X)^{\langle 2m\rangle} \iso
(\mathbb{P}X)^{\langle 2m-2\rangle}
\wedge_{\mathbb{P}(\bigvee_i S^{2m-1})}
\mathbb{P}(\bigvee_i D^{2m}).
\]
To calculate $K^\vee_*((\mathbb{P}X)^{\langle 2m\rangle})$
we may use a K\"unneth spectral sequence of the form
\begin{equation}\label{eq:KSS}
\mathrm{E}^2_{s,t} =
\Tor^{K^\vee_*(\mathbb{P}(\bigvee_i S^{2m-1}))}_{s,t}
(K^\vee_*((\mathbb{P}X)^{\langle 2m-2\rangle}),K_*)
\;\Lra\;
K^\vee_{s+t}((\mathbb{P}X)^{\langle 2m\rangle}),
\end{equation}
where the internal $t$ grading is in $\Z/2$, i.e.,
it is an integer modulo~$2$. This is essentially
described in~\cite{EKMM}, but we will require its
multiplicativity, and also the fact that it inherits
an action of power operations. The latter structure
is constructed in a similar fashion to the mod~$p$
Dyer-Lashof operations in~\cite{HL&IM}.
\begin{prop}\label{prop:KSS-collapsing}
The spectral sequence~\eqref{eq:KSS} collapses
at $\mathrm{E}^2$ to give
\[
K^\vee_{s+t}((\mathbb{P}X)^{\langle 2m\rangle})
=
K^\vee_{s+t}((\mathbb{P}X)^{\langle 2m-2\rangle})
     [\dlQ^sx_{i}:s\geq0,\; i\;]\sphat_p,
\]
where each $x_i$ is in even degree.
\end{prop}
\begin{proof}
Recall from~\cite{TB&MF} that
\[
K^\vee_*
\Biggl(\mathbb{P}\biggl(\bigvee_i S^{2m-1}\biggr)\Biggr)
               = \Lambda(z_i)\sphat_p,
\]
the $p$-completed exterior algebra on odd degree
generators
$z_i\in K^\vee_1(\mathbb{P}(\bigvee_i S^{2m-1}))$,
each of which originates on a wedge summand.

The $\mathrm{E}^2$-term is a divided power algebra
over $K^\vee_*((\mathbb{P}X)^{\langle 2m-2\rangle})$
on generators of bidegree $(1,1)$, each represented
in the cobar complex by $[\dlQ^sz_i]$. We will write
$\gamma_r([\dlQ^sz_i])$ for the $r$-th divided power
of this element and recall that the particular elements
$\gamma_{(r)}([\dlQ^sz_i])=\gamma_{p^r}([\dlQ^sz_i])$
generate the  algebra subject to relations of the form
\[
\gamma_{(r)}([\dlQ^sz_i])^p =
\binom{p^{r+1}}{p^r,\ldots,p^r}\gamma_{(r+1)}([\dlQ^sz_i]),
\]
where the multinomial coefficient satisfies
\[
\binom{p^{r+1}}{p^r,\ldots,p^r} = p t
\]
for some integer~$t$ not divisble by~$p$. For degree
reasons there can only be trivial differentials, so
the only issue still to be resolved is that of the
multiplicative structure.

We follow a line of argument similar to that of~\cite{HL&IM}.
In the spectral sequence we have
\[
\dlQ[z_i] = [\dlQ z_i],
\]
so it remains to relate this element to a $p$-th
power in the target of the spectral sequence.
By~\cite{LNM1176}*{chapter~IX, theorem~3.3(viii)},
if~$Z_i$ is represented by $[z_i]$, then
$Z_i^p+p\dlQ Z_i$ is represented by $[\dlQ z_i]$,
therefore $Z_i^p$ is represented by
\[
(1-p)[\tdlQ z_i]\equiv [\tdlQ z_i]\pmod{p}.
\]
It follows that each such $Z_i$ has non-trivial
$p$-th power also represented in the $1$-line.
By induction this can be extended to show that
each $\gamma_{(r)}([\dlQ^sz_i])$ represents an
element with non-trivial $p$-th power. Finally,
an easy argument shows that the target is a
completed polynomial algebra as stated.
\end{proof}

It is also useful to generalise this to the case
of a CW spectrum $Y$ with chosen $0$-cell $S^0\to Y$,
where $S^0\xrightarrow{\sim}S$ is the functorial
cofibrant replacement of $S$ in the model category
of $S$-modules. We may then consider the reduced
free commutative $S$-algebras $\widetilde{\mathbb{P}}Y$
which is defined as the homotopy pushout of the
diagram of solid arrows
\[
\xymatrix{
\mathbb{P}S^0\ar[r]\ar[d]\ar@{}[dr]|{\PO} & \mathbb{P}Y\ar@{.>}[d]  \\
S\ar@{.>}[r] & \widetilde{\mathbb{P}}Y
}
\]
where the vertical map is the canonical multiplicative
extension of $S^0\to S$; see~\cite{TAQI} for more on
this construction. As a particular case, we can consider
a map $f\:S^{2m-1}\to S^0$ and form its mapping cone
$C_f=S^0\cup_f D^{2m}$. Then take
$S/\!/f = \widetilde{\mathbb{P}}C_f$ to be a homotopy
pushout for the diagram
\[
\xymatrix{
\mathbb{P}S^0\ar[r]\ar[d]\ar@{}[dr]|{\PO} & \mathbb{P}C_f\ar@{.>}[d] \\
S\ar@{.>}[r] & S/\!/f
}
\]
and there is an associated K\"unneth spectral sequence
\begin{equation}\label{eq:KSS-S//f}
\mathrm{E}^2_{s,t} =
\Tor^{K^\vee_*(\mathbb{P}S^0)}(K_*,K^\vee_*(\mathbb{P}C_f))
\;\Lra\;
K^\vee_{s+t}(S/\!/f).
\end{equation}
It is easily seen that
\[
K^\vee_*(\mathbb{P}S^0) = \Z_p[\dlQ^sx_{0}:s\geq0]\sphat_p
\]
is a subalgebra of
\[
K^\vee_*(\mathbb{P}C_f)
      = \Z_p[\dlQ^sx_{0},\dlQ^sx_{2m}:s\geq0]\sphat_p,
\]
and the spectral sequence has
\[
\mathrm{E}^2_{0,*} =
K_*\otimes_{K^\vee_*(\mathbb{P}S^0)} K^\vee_*(\mathbb{P}C_f)
= \Z_p[\dlQ^sx_{2m}:s\geq0]\sphat_p,
\qquad
\mathrm{E}^2_{r,*} = 0 \quad(r\geq1).
\]
This discussion establishes
\begin{prop}\label{prop:S//f}
We have
\[
K^\vee_*(S/\!/f) = \Z_p[\dlQ^sx_{2m}:s\geq0]\sphat_p.
\]
\end{prop}

Provided we know the coaction for $K^\vee_*(C_f)$,
that for $K^\vee_*(S/\!/f)$ follows formally. In
general we have only the following possible form
of coaction,
\[
\Psi(x_{2m}) = w^m\otimes x_{2m} + c(f)(1-w^m),
\]
where $c(f)$ is a certain kind of rational number.
Then
\[
\Psi(\dlQ^s x_{2m}) = \dlQ^s(\Psi x_{2m})
\]
which involves iterated application of~$\dlQ$.

\section{Some examples based on elements of
Hopf invariant~$1$}\label{sec:MoreHopfInvt1}

Throughout this section we assume that $p=2$.

We will consider the examples $S/\!/\eta$
and $S/\!/\nu$ previously discussed in~\cite{Char}.
Similar considerations apply to other examples
constructed using elements in the image of the
$J$-homomorphism at an arbitrary prime. In order
to study these examples, it is necessary to
determine the $K^\vee_0K$-coaction on
$K^\vee_0(S/\!/f)$. Our goal is to explain
why the following algebraic results holds.
\begin{thm}\label{thm:eta-nu-sigma}
There are continuous epimorphisms of\/
$2$-complete $\Z_2$-$\theta$-algebras
\[
K^\vee_0(S/\!/\eta) \to K^\vee_0K,
\quad
K^\vee_0(S/\!/\nu) \to K^\vee_0K,
\]
where in each case the domain is a free
$\theta$-algebra. Moreover, these are
induced by morphisms of \Einfty{} ring
spectra $S/\!/\eta\to K$ and
$S/\!/\nu\to K$.
\end{thm}
\begin{proof}
We give the ingredients required for the
case of~$\eta$, the other being similar.

We will use the following elements
$\Phi_s=\Phi_s(w)$ ($s\geq0$) of $K^\vee_0K$:
\begin{equation}\label{eq:Phi}
\Phi_0 = \frac{(1-w)}{2},
\qquad
\Phi_n = \frac{(\Phi_{n-1}-\Phi_{n-1}^2)}{2}
           \quad (n\geq1).
\end{equation}
By results of~\cite{padic}, $K^\vee_0K$ has
a topological basis consisting of the monomials
\begin{equation}\label{eq:KK-TopBasis}
\Phi_0^{\epsilon_0}\Phi_1^{\epsilon_1}
                     \cdots\Phi_\ell^{\epsilon_\ell}
\quad (\epsilon_i=0,1).
\end{equation}
If we view these as continuous functions on
$\Z_{2}^\times$, then for a $2$-adic unit
$\alpha$ expressed as
\[
\alpha =
1 - (2a_0 + 2^2a_1 + \cdots + 2^{r+1}a_r + \cdots)
\]
with $a_r = 0,1$, in $\Z_{2}$ we have
\[
\Phi_r(\alpha) \equiv a_r \pmod{2}.
\]
We also know that $\dlQ\Phi_s = \Phi_{s+1}$,
hence $\Phi_s = \dlQ^s\Phi_0$.

In the case where $f=\eta$, we can take the
generator~$x_2$ to have coaction
\begin{equation}\label{eq:x2-coaction}
\Psi(x_2) = \Phi_0\otimes 1 + w\otimes x_2
          = \Phi_0 + wx_2,
\end{equation}
where we suppress the tensor product symbols
when the meaning seems clear without them.
For the coproduct in $K^\vee_0K$ we have
\[
\Psi\Phi_0 = \Phi_0\otimes1 + w\otimes\Phi_0,
\]
and also
\[
\Psi\dlQ x_2 =
w\dlQ x_2 + w\Phi_0 x_2^2 - w\Phi_0 x_2 + \Phi_1.
\]
Without further calculation we see that there is
a homomorphism of topological comodule algebras
\[
\Z_2[x_2]\sphat_2 \to K^\vee_0K;
             \quad x_2\mapsto\Phi_0.
\]
This is induced from a morphism of \Einfty{} ring
spectra $S/\!/\eta\to K$ arising from the fact
that the composition of $\eta\:S^1\to S$ with the
unit $S\to K$ is null homotopic. Therefore there
is an extension to a continuous epimorphism
\[
K^\vee_0(S/\!/\eta) \to K^\vee_0K;
 \quad
 \dlQ^sx_2 \mapsto \Phi_s.
\]
This displays $K^\vee_0K$ as a quotient of the
free $\theta$-algebra $K^\vee_0(S/\!/\eta)$ as
in Proposition~\ref{prop:K0K-ThetaAlg}.
\end{proof}

\begin{thm}\label{thm:S//eta-splitting}
There is a $K(1)$-local equivalence
\[
S/\!/\eta \xrightarrow{\;\sim\;}\prod_{j\geq0} K.
\]
\end{thm}
\begin{proof}[Outline of Proof]
We will use the homology theory $K(1)_*(-)$,
i.e., mod~$2$ $K$-theory. For the spectra we
are considering, odd degree groups are trivial
so we can consider the ungraded $\F_2$-vector
spaces obtained from $K(1)_0(-)$. This functor
takes values in the category of
$K(1)_0(K)$-comodules, where
$K(1)_0(K)\subseteq K(1)_0(K(1))$ is the
subHopf algebra called the Morava stabiliser
(Hopf) algebra and often denoted (rather
confusingly) $K(1)_0K(1)$ in the literature.

Using the basis of~\eqref{eq:KK-TopBasis},
we see that the group-like element
$w=1-2\Theta_0\in K^\vee_0(K)$ reduces
mod~$2$ to~$1$ and this is the only
group-like element of $K(1)_0(K)$.
The reductions mod~$2$ of this basis
give a basis for $K(1)_0(K)$ and the
increasing coradical filtration
$F_kK(1)_0(K)$ ($k\geq0$) defined by
Laures \& Schuster~\cite{GL&BS:K2localMString}*{section~2}
has
\[
F_kK(1)_0(K) = \F_2\{1,\Phi_0,\ldots,\Phi_{k-1}\}.
\]

The epimorphism $K^\vee_0(S/\!/\eta)\to K^\vee_0(K)$
gives rise to a commutative diagram of
$K(1)_0(K)$-comodule algebras of the
following shape.
\[
\xymatrix{
& K(1)_0(S/\!/\eta)\ar[d]_{\Psi}\ar@{->>}[r]\ar@{->>}@/_15pt/[ddl]
& K(1)_0(K)\ar[d]^{\Psi} \\
& K(1)_0(K)\otimes K(1)_0(S/\!/\eta)\ar@{->>}[r]\ar@{->>}[d]_{\mathrm{counit}\otimes\Id}
& K(1)_0(K)\otimes K(1)_0(K)\ar@{->>}@/^15pt/[dl]^{\mathrm{counit}\otimes\Id} \\
K(1)_0(K)\ar@{<->}[r]^(.45)\iso & \F_2\otimes K(1)_0(K) &
}
\]
The coaction for $K(1)_0(S/\!/\eta)=\F_2[\dlQ^s:s\geq0]$
is computable recursively, for example~\eqref{eq:x2-coaction}
gives
\[
\Psi(x_2) = \Phi_0\otimes 1 + 1\otimes x_2.
\]
and
\[
\Psi(\dlQ x_2) =
\Phi_1\otimes1 + \Phi_0\otimes(x_2+x_2^2) + 1\otimes\dlQ x_2.
\]

Now we can use an appropriate version of the
classic Milnor-Moore Theorem of~\cite{M&M:HopfAlg},
see for example Laures \& Schuster~\cite{GL&BS:K2localMString}*{theorem~2.8},
to deduce that
\[
K(1)_0(S/\!/\eta) \iso
K(1)_0(K)\otimes \Prim_{K(1)_0(K)}K(1)_0(S/\!/\eta),
\]
where $\Prim_{K(1)_0(K)}K(1)_0(S/\!/\eta)\subseteq K(1)_0(S/\!/\eta)$
is the subalgebra of primitives. To use this,
we need to determine filtration
\[
F_kK(1)_0(S/\!/\eta) =
\Psi^{-1}(F_k K(1)_0(K)\otimes K(1)_0(S/\!/\eta))
\quad(k\geq0)
\]
associated with the coradical filtration. By
induction we find that
\[
F_kK(1)_0(S/\!/\eta) = \F_2[x_2,\ldots,\dlQ^{k-1}x_2].
\]
We need to check the condition that the surjection
$K(1)_0(S/\!/\eta)\to K(1)_0(K)$ is a $\star$-isomorphism
as in~\cite{GL&BS:K2localMString}*{definition~2.6}
(note that as we are working with \emph{left}
comodules we need to consider graded \emph{right}
primitives). Using an induction on~$k$, we find
that the $k$-graded right primitive subspace is
$F_kK(1)_0(S/\!/\eta)$ and this maps onto
$F_kK(1)_0(K)$ which is the $k$-graded right
primitive subspace of~$K(1)_0(K)$.

Dualising and taking care with the inherent linearly
compact topologies and completed tensor products
involved, we obtain an isomorphism of left
topological $K(1)^0(K)$-modules
\[
K(1)^0(S/\!/\eta) \iso
K(1)^0(K)\hotimes(\Prim_{K(1)_0(K)}K(1)_0(S/\!/\eta))^\dagger,
\]
where $V^\dagger$ denotes the set of functionals
supported on finite dimensional subspaces of
the vector space~$V$. Choosing a topological
basis $\{b_\alpha:\alpha\in A\}$ for
$(\Prim_{K(1)_0(K)}K(1)_0(S/\!/\eta))^\dagger$,
we may lift each $b_\alpha$ to an element
$\tilde{b_\alpha}\in K^0(S/\!/\eta)$ since
$K(1)_1(S/\!/\eta)=0$. This gives a map
$S/\!/\eta\to \prod_{\alpha\in A} K$ which
induces a $K(1)$-isomorphism, hence it is
a $K(1)$-local equivalence. In fact~$A$
can be taken to be countable, so we might
as well index on the natural numbers.
%
%
\end{proof}

Notice that there is an \Einfty{} morphism $S/\!/\eta\to kU$
which induces a surjection on $\pi_*(-)$ but not on
$H_*(-;\F_2)$. Hence $kU$ cannot be a retract of $S/\!/\eta$
$2$-locally or after $2$-completion. However, multiplication
by the Bott map induces a cofibre sequence
\[
\Sigma^2 kU\to kU\to H\Z
\]
where $KU\wedge H\Z$ is rational. Therefore $\Sigma^2 kU\to kU$
is a $K(1)$-local equivalence, so it induces an isomorphism
on $K^\vee(-)$.

Notice that
\[
w^2 = (1 - 2\Phi_0)^2
    = 1 - 4(\Phi_0-\Phi_0^2)
    = 1 - 8\Phi_1,
\]
so
\[
1-w^2 = 8\Phi_1.
\]
Similarly,
\[
w^4 = 1 - 16(\Phi_1-\Phi_1^2) + 48\Phi_1^2,
\]
and therefore
\[
1 - w^4 = 16(\Phi_1-\Phi_1^2) - 48\Phi_1^2
        = 32\Phi_2 - 48\Phi_1^2.
\]
Such identities allow us to describe the groups
\[
\Ext^{1,2n}_{K_*K}(K_*,K_*) =
\Pr K_{2n}K/(\eta_{\mathrm{L}}-\eta_{\mathrm{R}})K_{2n}
\]
that detect the $2$-primary part of image of the
$J$-homomorphism through the $e$-invariant. Here~$\Pr$
denotes the subgroup of primitive elements which satisfy
\[
\Psi(x) = 1\otimes x + x\otimes1,
\]
and $\eta_{\mathrm{L}},\eta_{\mathrm{R}}$ denote the
left and right units respectively. When $n=1,2,4$,
these groups are cyclic with the following orders and
generators:
\begin{itemize}
\item
{\ }$2$, generator represented by $u\Phi_0$;
\item
{\ }$8$, generator represented by $u^2\Phi_1$;
\item
{\ }$16$, generator represented by $u^4(2\Phi_2-3\Phi_1^2)$.
\end{itemize}
Here we write $u\in K_2$ for the Bott generator. In the
first and last cases, a generator of $(\im J)_{2n-1}$
maps to the generator, but in the middle case only the
multiples of $2u^2\Phi_1$ are hit; for details
see~\cites{MRW,DCR:Novices}.

For $S/\!/\nu$ and $S/\!/\sigma$,
\[
K^\vee_0(S/\!/\nu) = \Z_2[\dlQ^sx_4:s\geq0]\sphat_2,
\quad
K^\vee_0(S/\!/\sigma) = \Z_2[\dlQ^sx_8:s\geq0]\sphat_2,
\]
we have the coactions
\[
\Psi x_4 = w^2\otimes x_4 + 2\Phi_1,
\quad
\Psi x_8 = w^4\otimes x_8 + 2\Phi_2 - 3\Phi_1^2.
\]

Finally, we note that there is an \Einfty{} morphism
$S/\!/\nu\to kO$ inducing an epimorphism on $\pi_*(-)$
which is not an epimorphism on $H_*(-;\F_2)$. The
composition $S/\!/\nu\to kO\to KO$ induces a $K(1)$-local
splitting whose proof is similar to that of
Theorem~\ref{thm:S//eta-splitting}.
\begin{thm}\label{thm:S//nu-splitting}
There is a $K(1)$-local equivalence
\[
S/\!/\nu \xrightarrow{\;\sim\;}
             \prod_{j\geq0}\Sigma^{4\rho(j)}KO,
\]
for some numerical function $\rho$ taking
values in~$\{0,1\}$.
\end{thm}

\begin{rem}\label{rem:sigma}
The case of $S/\!/\sigma$ should also be
amenable to a similar analysis, however
we have not found convenient way to
formalise an argument for this case.
\end{rem}


\begin{bibdiv}
\begin{biblist}

\bib{JFA:BlueBook}{book}{
   author={Adams, J. F.},
   title={Stable Homotopy and Generalised Homology},
   series={Chicago Lectures in Mathematics},
   note={Reprint of the 1974 original},
   publisher={University of Chicago Press},
   date={1995},
}

\bib{JFA&FWC}{article}{
   author={Adams, J.~F.},
   author={Clarke, F.~W.},
   title={Stable operations on complex $K$-theory},
   journal={Ill. J. Math.},
   volume={21},
   date={1977},
   pages={826\ndash829},
}

\bib{AHS}{article}{
   author={Adams, J. F.},
   author={Harris, A. S.},
   author={Switzer, R. M.},
   title={Hopf algebras of cooperations for
   real and complex $K$-theory},
   journal={Proc. London Math. Soc. (3)},
   volume={23},
   date={1971},
   pages={385\ndash408},
}
		
\bib{padic}{article}{
   author={Baker, A.},
   title={$p$-adic continuous functions
   on rings of integers},
   journal={J. Lond. Math. Soc.},
   volume={33},
   date={1986},
   pages={414\ndash20},
}

\bib{In-localEn}{article}{
   author={Baker, A.},
   title={$I_n$-local Johnson-Wilson spectra
   and their Hopf algebroids},
   journal={Documenta Math.},
   volume={5},
   date={2000},
   pages={351\ndash64},
}

\bib{L-complete}{article}{
   author={Baker, A.},
   title={$L$-complete Hopf algebroids and
   their comodules},
   journal={Contemp. Math.},
   volume={504},
   date={2009},
   pages={1\ndash22},
}

\bib{TAQI}{article}{
    author={Baker, A.},
    title={Calculating with topological Andr\'e-Quillen
          theory, I: Homotopical properties of universal
          derivations and free commutative $S$-algebras},
    eprint={arXiv:1208.1868 (v5+)},
      date={2012},
}

\bib{Nishida}{article}{
   author={Baker, A.},
   title={Power operations and coactions
   in highly commutative homology theories},
   journal={Publ. Res. Inst. Math. Sci.},
   volume={51},
   pages={237\ndash272},
   date={2015},
}

\bib{Char}{article}{
    author={Baker, A.},
    title={Characteristics for $E_\infty$
    ring spectra},
    journal={Contemp. Math.},
    pages={1\ndash17},
    volume={708},
    date={2018},
}

\bib{HGamma}{article}{
   author={Baker, A.},
   author={Richter, B.},
   title={On the $\Gamma$-cohomology of rings
   of numerical polynomials and $E_\infty$
   structures on $K$-theory},
   journal={Commentarii Math. Helv.},
   volume={80},
   date={2005},
   pages={691\ndash723},
}

\bib{TB&MF}{article}{
   author={Barthel, T.},
   author={Frankland, M.},
   title={Completed power operations for
   Morava $E$-theory},
   journal={Algebr. Geom. Topol.},
   volume={15},
   date={2015},
   pages={2065\ndash2131},
}

\bib{AKB:padiclambda}{article}{
   author={Bousfield, A.~K.},
   title={On $p$-adic $\lambda$-rings
   and the $K$-theory of $H$-spaces},
   journal={Math. Z.},
   volume={223},
   date={1996},
   number={3},
   pages={483\ndash519},
}

\bib{LNM1176}{book}{
   author={Bruner, R. R.},
   author={May, J. P.},
   author={McClure, J. E.},
   author={Steinberger, M.},
   title={$H_\infty $ ring spectra and
   their applications},
   series={Lect. Notes in Math.},
   volume={1176},
   date={1986},
}

\bib{FWC:UltrametAnal}{article}{
   author={Clarke, F. W.},
   title={$p$-adic analysis and operations
   in $K$-theory},
   journal={Groupe de travail d'analyse
   ultram\'etrique},
   volume={14},
   number={15},
   date={1986\ndash7},
   eprint={http://www.numdam.org/item?id=GAU_1986-1987__14__A7_0},
}

		
\bib{TMF}{book}{
   author={Douglas, D. L.},
   author={Francis, J.},
   author={Henriques, A. G.},
   author={Hill, M. A.},
   title={Topological Modular Forms},
   series={Math. Surv. and Mono.},
   volume={201},
   date={2014},
}

\bib{EKMM}{book}{
    author={Elmendorf, A. D.},
    author={\Kriz, I.},
    author={Mandell, M. A.},
    author={May, J. P.},
    title={Rings, modules, and algebras in
    stable homotopy theory},
    journal={Math. Surv. and Monographs},
    volume={47},
    note={With an appendix by M. Cole},
    date={1997},
}

\bib{JPG&JPM:I-adic}{article}{
   author={Greenlees, J. P. C.},
   author={May, J. P.},
   title={Derived functors of $I$-adic
   completion and local homology},
   journal={J. Algebra},
   volume={149},
   date={1992},
   pages={438\ndash453},
}

\bib{MJH:K(1)localEinfty}{article}{
    author={Hopkins, M. J.},
    title={$K(1)$-local $E_\infty$-ring spectra},
    journal={Math. Surv. and Mono.},
    pages={287\ndash302},
    volume={201},
    date={2014},
}

\bib{Hovey:L-colimits}{article}{
   author={Hovey, M.},
   title={Morava $E$-theory of filtered colimits},
   journal={Trans. Amer. Math. Soc.},
   volume={360},
   date={2008},
   number={1},
   pages={369\ndash382},
}

\bib{MH:SSMoravaEthy}{article}{
    author={Hovey, M.},
    title={Some spectral sequences in Morava $E$-theory},
    eprint={http://mhovey.web.wesleyan.edu},
}

\bib{MAMS666}{article}{
   author={Hovey, M.},
   author={Strickland, N. P.},
   title={Morava $K$-theories and localisation},
   journal={Mem. Amer. Math. Soc.},
   volume={139},
   date={1999},
   number={666},
}

\bib{MJ:Glasgow}{article}{
   author={Joachim, M.},
   title={Higher coherences for equivariant $K$-theory},
   journal={Lond. Math. Soc. Lect. Note Ser.},
   volume={315},
   date={2004},
   pages={87\ndash114},
}

\bib{SOK:DLops}{article}{
   author={Kochman, S. O.},
   title={Homology of the classical groups
   over the Dyer-Lashof algbera},
   journal={Trans. Amer. Math. Soc.},
   volume={185},
   date={1973},
   pages={83\ndash136},
}

\bib{GL&BS:K2localMString}{article}{
   author={Laures, G.},
   author={Schuster, B.},
   title={Towards a splitting of the $K(2)$-local
   string bordism spectrum},
   journal={Proc. Amer. Math. Soc.},
   volume={147},
   date={2019},
   pages={399\ndash410},
}

\bib{HL&IM}{article}{
   author={Ligaard, H.},
   author={Madsen, I.},
   title={Homology operations in the
   Eilenberg-Moore spectral sequence},
   journal={Math. Z.},
   volume={143},
   date={1975},
   pages={45\ndash54},
}

\bib{MRW}{article}{
   author={Miller, H. R.},
   author={Ravenel, D. C.},
   author={Wilson, W. S.},
   title={Periodic phenomena in the Adams-Novikov
   spectral sequence},
   journal={Ann. of Math. (2)},
   volume={106},
   date={1977},
   number={3},
   pages={469\ndash516},
}

\bib{M&M:HopfAlg}{article}{
   author={Milnor, J. W.},
   author={Moore, J. C.},
   title={On the structure of Hopf
   algebras},
   journal={Ann. of Math.},
   volume={81},
   date={1965},
   pages={211\ndash264},
}

		
\bib{DCR:Novices}{article}{
   author={Ravenel, D. C.},
   title={A novice's guide to the Adams-Novikov
   spectral sequence},
   journal={Lect. Notes in Math.},
   volume={658},
   date={1978},
   pages={404\ndash475},		
}


\bib{CR:CongCond}{article}{
   author={Rezk, C.},
   title={The congruence criterion for
   power operations in Morava $E$-theory},
   journal={Homology Homotopy Appl.},
   volume={11},
   date={2009},
   pages={327\ndash379},
}

\bib{CR:ModIsogCmplxes}{article}{
   author={Rezk, C.},
   title={Modular isogeny complexes},
   journal={Algebr. Geom. Topol.},
   volume={12},
   date={2012},
   pages={1373\ndash1403},
}

\bib{CR:PowOps-Koszul}{article}{
    author={Rezk, C.},
    title={Rings of power operations
    for Morava $E$-theories are Koszul},
    eprint={arXiv:1204.4831},
    date={2017},
}

\end{biblist}
\end{bibdiv}
\end{document}